\documentclass[review]{elsarticle}

\usepackage{amsthm,amsmath,amsfonts,amssymb}
\usepackage{subcaption}
\usepackage{lineno, hyperref}
\modulolinenumbers[5]
\usepackage{color}
\usepackage[utf8]{inputenc}

\usepackage{mathrsfs}

\newtheorem{theorem}{Theorem}[section]

\newtheorem{remark}[theorem]{Remark}

\newcommand{\vertiii}[1]{{\left\vert\kern-0.25ex\left\vert\kern-0.25ex\left\vert #1
\right\vert\kern-0.25ex\right\vert\kern-0.25ex\right\vert}}
\newcommand{\be}{\begin{equation}\begin{aligned}}
\newcommand{\ee}{\end{aligned}\end{equation}}
\newcommand{\ben}{\begin{equation}\nonumber\begin{aligned}}

\def\e1{{\varepsilon_{11}}}
\def\b1{{\beta_{11}}}
\def\bp3{{\beta_{33}}}
\def\ep3{{\varepsilon_{33}}}

\begin{document}

\begin{frontmatter}	

\title{Mathematical analysis and numerical simulation of the \\ Guyer–Krumhansl heat equation}

\author[address1]{A. J. A. Ramos}
\ead{ramos@ufpa.br}

\author[address3,address4]{R. Kovács\corref{mycorrespondingauthor}}
\ead{kovacs.robert@wigner.hu}

\author[address1]{M. M. Freitas}
\ead{mirelson@ufpa.br}

\author[address2]{D. S. Almeida Júnior}
\cortext[mycorrespondingauthor]{Corresponding author}
\ead{dilberto@ufpa.br}

\address[address1]{Faculty of Mathematics, Federal University of Par\'a, Raimundo Santana Street s/n, Salin\'opolis--PA, 68721-000, Brazil}

\address[address3]{Department of Energy Engineering, Faculty of Mechanical Engineering, Budapest University of Technology and Economics, Muegyetem rkp. 3., H-1111 Budapest, Hungary}

\address[address4]{Department of Theoretical Physics, Wigner Research Centre for Physics, Budapest, Hungary}

\address[address2]{PhD Program in Mathematics, Federal University of  Par\'a, Augusto Corr\^ea Street 01,
Bel\'em--PA, 66075-110, Brazil}

\begin{abstract}
The Guyer-Krumhansl heat equation has numerous important practical applications in both low-temperature and room temperature heat conduction problems. In recent years, it turned out that the Guyer-Krumhansl model can effectively describe the thermal behaviour of macroscale heterogeneous materials. Thus, the Guyer-Krumhansl equation is a promising candidate to be the next standard model in engineering. However, to support the Guyer-Krumhansl equation's introduction into the engineering practice, its mathematical properties must be thoroughly investigated and understood.
In the present paper, we show the basic structure of this particular heat equation, focusing on the differences in comparison to the Fourier heat equation {obtained when $(\tau_{q}, \mu^{2})\rightarrow (0,0)$}. Additionally, we prove the well-posedness of a particular, practically significant initial and boundary value problem. The stability of the solution is also investigated in the discrete space using a finite difference approach.
\end{abstract}
\vspace{0.5cm}

\end{frontmatter}

\section{Introduction}
Among the various heat conduction models, the Guyer-Krumhansl (GK) equation has an exceptional background and possesses significant potential in engineering practice. It is first derived on the basis of kinetic theory and applied for low-temperature problems \cite{GK66}. It served as a starting point to find the celebrated window condition, i.e., the relation that helped the researchers detect the second sound in solids, too \cite{GK64}.
The kinetic approach is limited due to the prescribed heat conduction mechanisms with phonons \cite{DreStr93a}. Therefore, one cannot apply such a framework to more general room temperature macroscale problems. However, there is the possibility of utilizing a continuum approach -- called internal variables -- to derive the Guyer-Krumhansl equation based on irreversible thermodynamics \cite{VanFul12}. Note that different continuum approaches can also be suitable such as GENERIC \cite{Grmela2018b, SzucsEtal21} and Extended Irreversible Thermodynamics \cite{Lebon89, JouEtal10b}. A continuum approach results in a model with a notably wider range of validity since it allows to fit the parameters to the observed phenomenon \cite{FehEtal21}.

In this work we consider the one-dimensional version of the GK equation, \textit{i.e.},
\begin{align}
\rho c T_{t}+q_{x}&=0 \quad \mbox{in} \quad (0,l) \times (0,\infty), \label{Eq1}
\\
\tau_{q}q_{t}+q-\mu^{2}q_{xx}+kT_{x}&=0 \quad \mbox{in} \quad (0,l) \times (0,\infty), \label{Eq2}
\end{align}
subject to boundary conditions
{\begin{align}
q(0,t)=q(l,t)=0, \quad \mbox{for all} \quad t\geq 0, \label{CCont1}
\end{align}}
and initial data
\begin{align}
&T(x,0)=T_{0}(x), \quad q(x,0)=q_{0}(x),\quad x \in (0,l), \label{CInicial}
\end{align}
wherein Eq.~\eqref{Eq1} is the internal energy balance with $\rho$ and $c$ being the mass density and specific heat, and no source terms are considered in a rigid conductor. Furthermore, $T$ denotes the temperature, $q$ is the heat flux, and the indicies $t$ and $x$ denote the corresponding derivatives respect to time and space. Eq.~\eqref{Eq2} is a constitutive equation in which $\tau_q,\mu^2$ and $k$ are positive coefficients following from the second law of thermodynamics but no further restrictions are acting on them in the linear regime (T-independent coefficients). Consequently, that continuum model can be compatible with kinetic theory if it adapts the particular form for the coefficients, if necessary \cite{KovVan15}. Thanks to the flexibility of the coefficients, it was possible to explain and model several room temperature experiments successfully without any need for phonon hydrodynamics \cite{FehEtal21, Botetal16, VanEtal17}. All these experiments were performed using a heterogeneous material such as rocks, foams and layered samples. In that sense, the GK model is promising and could be the next standard model in the engineerint practice in the future. However, it would not be possible without thoroughly investigating and understanding its mathematical properties first.

Therefore, in the following, we present an essential attribute of the coefficients, characteristic only for non-Fourier heat equations, which is a rarely known feature of the GK equation. Namely, the coefficients are functionally connected, which becomes crucial for T-dependent material parameters \cite{KovRog20}.
Then, considering only the linear problem, we prove the well-posedness of the model \eqref{Eq1}--\eqref{CInicial}. This is also crucial since there are popular models in the corresponding literature (such as the famous dual-phase-lag equation), which are not necessarily well-posed due to the missing physical and mathematical consistency \cite{Dreetal09, Ruk14, Ruk17}.

\section{Temperature-dependent coefficients}
In the framework of irreversible continuum thermodynamics, the constitutive equations are derived utilizing the second law of thermodynamics. Mathematically speaking, the constitutive equations are found as the solutions of the entropy inequality\begin{align}
\rho s_t + J_x = \sigma_s \geq 0, \label{sbal}
\end{align}
where $s$ is the entropy density, $J$ is the flux of the entropy density and $\sigma_s$ stands for the entropy production. The solutions of Eq.~\eqref{sbal} are called Onsagerian relations. The detailed derivation procedure can be found in \cite{VanFul12}. In the case of the GK equation, the Onsagerian relation is
\begin{align}
\left (\frac{1}{T}+l_1 q_x \right)_x - \rho m q_t -l_2 q=0, \label{GKE}
\end{align}
with $l_1,m \geq0$, $l_2>0$, furthermore
\begin{align}
\tau_q=\frac{\rho m}{l_2}, \quad k=\frac{1}{l_2 T^2}, \quad \mu^2 = \frac{l_1}{l_2} \label{coeffgk}
\end{align}
forming the coefficients of the GK constitutive equation, which reads
\begin{align}
\tau_q q_t + q - \mu^2 q_{xx}+k T_x=0.
\label{GK}
\end{align}
Eq.~\eqref{GK} is valid only for constant coefficients, i.e., when $l_1, l_2$ and $m$ do not depend on the state variable T, which is also apparent from Eqs.~\eqref{GKE}--\eqref{coeffgk}. Notably, Eq.~\eqref{coeffgk} presents the functional connections between the coefficients. In case of a temperature-dependent thermal conductivity $k(T)$, $l_2$ becomes a function of $T$, too, consequently both $\tau_q$ and $\mu^2$ become $T$-dependent. This is a non-trivial consequence of thermodynamics and characteristic only for non-Fourier heat equations.
As $k(T)$ is usual in the engineering practice, and Eq.~\eqref{coeffgk} is the direct consequence of the second law of thermodynamics, it is an essential physical property. That attribute is entirely missing from the dual-phase-lag approach.

Furthermore, the $T$-dependency of $\tau_q$ introduced by $l_2(T)$ is not necessarily the same, which is observed in the experiments for $\tau_q(T)$. In other words, one must adjust the $T$-dependence of $\tau_q$ according to the experiments using $\rho$ or $m$. Nevertheless, it has further consequences. If $\rho$ is $T$-dependent, then mechanics must be involved as well \cite{KovRog20}. If $m$ becomes a function, then the entire derivation must be repeated since Eq.~\eqref{GKE} is valid only for constant $m$. Either way, Eq.~\eqref{GKE} loses its validity, and one must be cautious.
Additionally, due to $l_2(T)$, $\mu^2$ is also $T$-dependent. If one has direct measurements about $\mu^2(T)$, then $l_1$ can be used to adjust the proper $T$-dependent behaviour. Following Eq.~\eqref{GKE}, $l_1(T)$ would lead to a new term, i.e.,
\begin{align}
\left (\frac{1}{T} \right)_x +\big(l_1(T)\big)_x q_x +l_1(T) q_{xx} - \rho m q_t -l_2 q=0, \label{GKE2}
\end{align}
holds, where
\begin{align}
\big(l_1(T)\big)_x q_x = \frac{\textrm{d} l_1(T)}{\textrm{d} T} T_x q_x,
\end{align}
and
\begin{align}
\frac{1}{l_2}\left (-\frac{1}{T^2} +\frac{\textrm{d} l_1(T)}{\textrm{d} T} q_x \right) T_x +\frac{l_1(T)}{l_2} q_{xx} - \frac{\rho m}{l_2} q_t -q=0. \label{GKE3}
\end{align}
Hence, beyond the $T$-dependence of $\mu^2(T)$, $l_1(T)$ introduces a new term, which could modify the thermal conductivity, leading to state dependence in $k$, seemingly. Furthermore, the positivity of $k$ could provide further restrictions on $l_1(T)$ since it appears with opposite sign in Eq.~\eqref{GKE3}.
Interestingly, $k$ is not simply $T$-dependent but also begins to depend on the spatial derivative (divergence) of $q$. It stands as an experimental challenge to distinguish this effect from the usual $T$-dependence of thermal conductivity, that is, by introducing $l_2(T)$. These have significantly different consequences and physical implications. 

Consequently, to determine the thermal conductivity of a material with complex inner structure, it is strongly suggested to apply only steady-state measurement methods. In steady-state, the equilibrium is not modified by the non-Fourier terms, and that approach could provide a reliable method to separate the $T$-dependence of these coefficients.

Overall, if the thermal conductivity depends on the temperature through $l_2(T)$, then the other parameters inherit that dependence, and one must include mechanics in the model because of $\rho(T)$. Furthermore, if $\mu^2$ depends on the temperature through $l_1(T)$, then it influences the thermal conductivity as well but leaves the relaxation time $\tau_q$ unchanged.

\section{Well-posedness}
Among the numerous non-Fourier heat conduction models, the dual-phase-lag (DPL) concept, i.e.,
\begin{align}
q(x,t+\tau_q) = - k T_x(x,t+\tau_T) \label{dpl}
\end{align}
stands out in popularity. Here, it is assumed that both quantities possess some time delay, and various models are derived by Taylor series expansion. Although the DPL concept is used to model various heat conduction problems, such an approach hides multiple substantial physical and mathematical aspects. The Eq.~\eqref{dpl} has no immediate compatibility with the first and second laws of thermodynamics, and the DPL model needs additional conditions. In other words, neither the energy conservation nor the stability is guaranteed, contrary to the Guyer-Krumhansl equation. It has been shown that Eq.~\eqref{dpl} could lead to ill-posed problems in certain situations \cite{Dreetal09, Ruk14, Ruk17}.
Therefore, it is of crucial importance to have a model with a solid mathematical and physical basis in order to avoid such stability or ill-posedness issues. 

Here, in the following, we show that well-posedness is fulfilled for the GK equation in the most crucial situation from a physical point of view. Let us recall that in Eqs.~\eqref{CCont1} and \eqref{CInicial} we use adiabatic (perfect thermal insulation) boundary conditions for inhomogeneous initial conditions. The other possibilities, the prescribed temperature and heat convection (i.e., the usual first and third-type) boundary conditions, would implement dissipation into the system. These could result in stable and convergent equilibrium solutions even for unstable and physically inconsistent models. In our case, the adiabatic boundary condition means that the initial energy content of the entire system remains the same and, in equilibrium, becomes homogeneous in space. This holds for the internal energy for which we have a balance equation, and it is a conserved quantity. 

However, from a mathematical point of view, energy as a notion is used differently, usually referring to a variational background, and here we follow the usual convention. Therefore, let us introduce the term `functional energy' of the system \eqref{Eq1}--\eqref{CInicial} as
\begin{align}\label{energy}
E(t):=\frac{\rho c}{2}\int_{0}^{l}T^{2}dx+\frac{\tau_{q}}{2k}\int_{0}^{l}q^{2}dx, 
\end{align}
which decreases in time as $q$ will be uniformly zero in equilibrium, i.e.,
\begin{equation}\label{dissipation.law}
\frac{\textrm{d}}{\textrm{d}t}E(t)=-\frac{1}{k}\int_{0}^{l}q^{2}\textrm{d}x-\frac{\mu^{2}}{k}\int_{0}^{l}q_{x}^{2}\textrm{d}x<0, \quad \mbox{for all} \quad t\geq 0,
\end{equation}
consequently, $E$ is a non-increasing function of the time variable $t$.

In this section we will show that the system \eqref{Eq1}--\eqref{CInicial} is well-posed using the semigroup technique. The following notations will be used in this paper:
\begin{eqnarray}
\big<u, v\big>:=\big<u, v\big>_{L^{2}(0, l)}
\quad \mbox{and} \quad \|u\|:=\sqrt{\big<u, u\big>}, \quad \mbox{for all} \quad u, v \in L^{2}(0,l).
\end{eqnarray}
Let us start by introducing the phase space by
\begin{eqnarray}\label{Hilbert_space}
\mathcal{H}:=L^2(0,l)\times L^2(0,l),
\end{eqnarray}
which is a Hilbert space endowed with the norm
\begin{eqnarray}
\|U\|^2_{\cal H}=\rho c\|u\|^2+\frac{\tau_{q}}{k}\|v\|^2
\end{eqnarray}
and the corresponding inner product $\big<\cdot,\cdot\big>_{\cal H}$, for all $U = (u,v)^{\top}\in \mathcal{H}$. Thus, denoting $u:=T$ and $v:=q$ we can transform the system (\ref{Eq1})--(\ref{CInicial}) into the first-order set of equations
\begin{eqnarray}\label{prob.Cauchy}
\begin{cases}
U_t(t) = {\cal A}U(t), \quad  \mbox{for all} \quad t>0,\label{prob.Cauchy.1} 
\\
U(0) = U_0,\label{prob.Cauchy.2} 
\end{cases}
\end{eqnarray}
where $U_0=(T_{0}, \, q_{0})^{\top}$ and $\cal A: D(\cal A)\subset \cal H\rightarrow \cal H$ is the differential operator given by
\begin{eqnarray}\label{Operador-A}
\mathcal{A}:=\left(\begin{array}{ccccc}
0 & -\frac{1}{\rho c}\frac{\partial(\cdot)}{\partial x} 
\\
\\
-\frac{k}{\tau_{q}}\frac{\partial(\cdot)}{\partial x} & -\frac{1}{\tau_{q}}I_{d}(\cdot)+ \frac{\mu^{2}}{\tau_{q}}\frac{\partial^{2}(\cdot)}{\partial x^{2}}  
\end{array}\right),
\end{eqnarray}
where $I_{d}(\cdot)$ denotes the identity operator, furthermore
\begin{eqnarray}
\nonumber D(\mathcal{A}):=\Big\{U=(u,v)^{\top} \in\mathcal{H}; \ \ \mu^{2}v_{x}-ku  \in  H^{1}(0,l), \ v\in H^{1}_{0}(0,l)\Big\}.
\end{eqnarray}

\begin{theorem}\label{WP}
The system \eqref{prob.Cauchy} is well-posed, \textit{i.e.}, for any $U_{0}\in \mathcal{H}$, the system \eqref{prob.Cauchy} has a unique weak solution $U(t)=e^{\mathcal{A}t}U_{0}\in C\big([0,\infty\big); \, \mathcal{H})$. Furthermore, if
$U_{0}\in D(\mathcal{A})$, $U(t)\in C^{1}\big([0,\infty); \, \mathcal{H}\big)\cap C\big([0,\infty); \, D(\mathcal{A})\big)$ becomes the classic solution for \eqref{prob.Cauchy}.
\end{theorem}

\begin{proof}
It is not difficult to see that $\mathcal{A}$ is a dissipative operator in the space $\mathcal{H}$. More precisely, we have
\begin{eqnarray}\label{part.real}
\big\langle\mathcal{A}U, U\big\rangle_{\mathcal{H}}
=-\frac{1}{k}\|q\|^{2}-\frac{\mu^{2}}{k}\|q_{x}\|^{2}<0.
\end{eqnarray}
Next, we prove that $\lambda I -  \mathcal{A}$ is surjective, for all $\lambda > 0$. Since $\mathcal{A}$ is dissipative and $D(\mathcal{A})$ is dense in $\mathcal{H}$, it is sufficient to show that $\mathcal{A}$ is maximal, that is, $\lambda I -  \mathcal{A}$ is surjective. For this purpose, given $F=\big(f_1, f_2\big)^{\top} \in \mathcal{H}$, we seek $U=\big(u, v\big)^{\top} \in D(\mathcal{A})$ which is the solution of
\begin{eqnarray}
(\lambda I - \mathcal{A})U = F,\label{2:0}
\end{eqnarray}
that is, the entries of $U$ satisfy the system of equations
\begin{eqnarray}
\lambda\rho c \,u+v_{x}=\rho cf_{1} \ &\in& \ L^2(0,l),\label{eq.01}
\\
\lambda\tau_{q}v+v-\mu^{2}v_{xx}+ku_{x}=\tau_{q}f_{2} \ &\in& \ L^2(0,l).\label{eq.02} 
\end{eqnarray}
Solving the system (\ref{eq.01})--(\ref{eq.02}) is equivalent to finding
\begin{equation*}
(u, v) \in L^{2}(0,l)\times H^{1}_{0}(0,l),
\end{equation*} 
such that 
\begin{eqnarray}
\lambda\rho c \,\big<u, \, \widetilde{u}\big>+\big<v_{x}, \, \widetilde{u}\big>&=&\rho c\, \big<f_{1},\, \widetilde{u}\big>,\label{eq.01b}
\\
\frac{\lambda\tau_{q}}{k}\big<v, \,\widetilde{v} \big>+\frac{1}{k}\big<v, \, \widetilde{v}\big>+\frac{1}{k}\big<\mu^{2}v_{x}-ku, \, \widetilde{v}_{x}\big>&=&\frac{\tau_{q}}{k}\big<f_{2}, \,\widetilde{v}\big>,\label{eq.02b} 
\end{eqnarray}
for all $(\widetilde{u}, \widetilde{v}) \in L^{2}(0,l)\times H^{1}_{0}(0,l)$.

Now we observe that solving the system (\ref{eq.01})--(\ref{eq.02}) is equivalent to solve the problem
\begin{equation}\label{2:12}
\mathcal{B}\big( (u,v),\left(\widetilde{u},\widetilde{v}\right) \big) = \mathcal{G}\left( \widetilde{u},\widetilde{v} \right),
\end{equation}
where $\mathcal{B}: \big[L^{2}(0,l)\times H^{1}_{0}(0,l)\big]^2 \rightarrow \mathbb{R}$ is the bilinear form given by 
\begin{align}
\mathcal{B} \big((u,v),\big(\widetilde{u},\widetilde{v}\big) \big) :=  \lambda\rho c \,\big<u, \, \widetilde{u}\big>+\frac{1}{k}(1+\lambda\tau_{q})\big<v, \,\widetilde{v} \big>+\big<v_{x}, \, \widetilde{u}\big>+\frac{1}{k}\big<\mu^{2}v_{x}-ku, \, \widetilde{v}_{x}\big>
\end{align}
and $\mathcal{G}: L^{2}(0,l)\times H^{1}_{0}(0,l) \rightarrow \mathbb{R}$ is the linear form given by
\begin{equation}\label{2:13}
\mathcal{G}\big(\widetilde{u}, \, \widetilde{v}\big) := \rho c\, \big<f_{1},\, \widetilde{u}\big>+\frac{\tau_{q}}{k}\big<f_{2}, \,\widetilde{v}\big>.
\end{equation}
Now, we introduce the Hilbert space $\mathcal{V} := L^{2}(0,l)\times H^{1}_{0}(0,l)$ equipped with the norm
\begin{equation}\label{2:14}
\big\|(u,v)\big\|_\mathcal{V}^2 = \|u\|^{2} + \|v\|^{2} + \mu^{2}\|v_{x}\|^{2}.
\end{equation}
It is clear that $\mathcal{B}$ and $\mathcal{G}$ are bounded. Furthermore, we can obtain that there exists a positive constant $C$ such that
\begin{eqnarray}
\mathcal{B} \big((u, v), \, (u, v) \big) &=& \lambda\rho c \| u\|^2+ \frac{1}{k}(1+\lambda\tau_{q})\| v\|^2 + \frac{\mu^{2}}{k}\| v_{x}\|^2  \nonumber 
\\
&\geq& C\|(u, v)\|_\mathcal{V}^2, \label{2:15}
\end{eqnarray}
which implies that $\mathcal{B}$ is coercive.

Hence, we assert that $\mathcal{B}$ is continuous and coercive form on $\mathcal{V}\times\mathcal{V}$ and $\mathcal{G}$ is continuous form on $\mathcal{V}$. So applying the Lax-Milgram Theorem, we deduce that for all
\begin{equation}\label{2:16}
(\widetilde{u}, \widetilde{v}) \in  L^{2}(0,l)\times H^{1}_{0}(0,l)
\end{equation}
the problem \eqref{2:12} admits a unique solution
\begin{equation}\label{2:17}
(u, v) \in  L^{2}(0,l)\times H^{1}_{0}(0,l).
\end{equation}
{
Furthermore, it follows from Eqs. (\ref{eq.01}) and (\ref{eq.02}) that $\mu^{2}v_{x}-ku \in H^{1}(0,l)$.
}

Thus the operator $\lambda I -  \mathcal{A}$ is surjective for all $\lambda>0$. Consequently, $\mathcal{A}$ is a maximal operator. Hence, the result of Theorem \ref{WP} follows from Lumer–Phillips Theorem (see \cite{Pazy}).
\end{proof}

{
\section{Uniform stabilization}

In this section we prove that the solution of the system (\ref{Eq1})--(\ref{CInicial}) decay exponentially uniformly with respect to the parameters $\tau_{q}$ and $\mu^{2}$ to an equilibrium point. In the $(\tau_{q}, \mu^{2})\rightarrow (0,0)$ limit, we obtain  Fourier model for which the energy decays exponentially as well in agreement with the thermodynamic requirements. Furthermore, we prove that the equilibrium point (or state) depends directly on the space where the initial condition $T_{0}$ is inserted.

\begin{theorem}\label{main.result.exponential.decay} Let $E(t)$ be the functional energy  for the system (\ref{Eq1})--(\ref{CInicial}). Suppose that the initial data $(T_{0}, q_{0})\in D(\mathcal{A})$.
Then, there exist constants $M_{0}$, $M_{1}$, $\omega>0$ uniformly bounded as $(\tau_{q}, \mu^{2})\rightarrow (0,0)$, such that $E(t)$ satisfies the following decay estimates:
\begin{itemize}
\item[\textbf{$(i)$}] If the initial data $T_{0}\notin L_{*}^{2}(0,l):=\big\{\mathcal{T}\in L^{2}(0,l); \ \int_{0}^{l}\mathcal{T}dx=0\big\}$, then
\begin{eqnarray}\label{decay.new.energy}
E(t)\leq M_{0}E(0)e^{-\omega t}+M_{1}\|C_{T}\|_{\infty},\quad \mbox{for all} \quad t\geq 0,
\end{eqnarray}
where $C_{T}(t):=\big(\mu^{2}q_{x}(0,t)-kT(0,t)\big)\int_{0}^{l}T(s,t)ds$ and $\|C_{T}\|_{\infty}:=\sup\big\{|C_{T}(t)|; \ t\in \mathbb{R}^{+}\big\}$ uniformly bounded as $\mu^{2}\rightarrow 0$.
\item[\textbf{$(ii)$}] If the initial data $T_{0}\in L_{*}^{2}(0,l)$, then
\begin{eqnarray}\label{decay.new.energy}
E(t)\leq
M_{0} E(0)e^{-\omega t}, \quad \mbox{for all} \quad t\geq 0.
\end{eqnarray}
\end{itemize}
\end{theorem}
{
From a physical point of view, Theorem \ref{main.result.exponential.decay} separates two cases of initial conditions. The first item, \textbf{$(i)$}, describes a physically realistic situation in which the initial total energy content of the system (integration of the initial temperature distribution) is not zero, therefore the temperature history must converge to the equilibrium prescribed by that energy content for an insulated (thermally closed) system. 

The second part, \textbf{$(ii)$}, however, is only for the mathematical completeness, since that situation is not physically realizable. The temperature - when calculating the energy content - must be measured in Kelvin degrees, therefore the zero initial energy content can be achieved only by uniformly zero temperature distribution, otherwise one implies negative absolute temperature. In the following, both parts are proved, but we emphasize that only part  \textbf{$(i)$} has practical importance.}

\begin{proof}
Integrating Eq. (\ref{Eq1}) over $[x,l]\subset[0,l]$, multiplying by $\int_{x}^{l}T(s,t)ds$ and integrating over $ [0,l] $ we obtain	
\begin{eqnarray}
\rho c \ \big<\int_{x}^{l}T_{t}(s,t)\textrm{d}s, \, \int_{x}^{l}T(s,t)\textrm{d}s\big>-\big<q, \, \int_{x}^{l}T(s,t)\textrm{d}s\big>=0. \label{p1}
\end{eqnarray}	
On the other hand, multiplying Eq. $(\ref{Eq2})$ by $\int_{x}^{l}T(s,t)\textrm{d}s$ and integrating over $ [0,l] $ we have
\begin{eqnarray}	
\tau_{q}\big<q_{t}, \,\int_{x}^{l}T(s,t)\textrm{d}s\big>+\big<q, \,\int_{x}^{l}T(s,t)\textrm{d}s\big>+\big(\mu^{2}q_{x}(0,t)-kT(0,t)\big)\int^{l}_{0}T(s,t)\textrm{d}s-\big<\mu^{2}q_{x}-kT, \,T\big>=0. \label{p2}
\end{eqnarray}	
Adding Eqs.~\eqref{p1} and \eqref{p2}, and considering that
$\big<\int_{x}^{l}T_{t}(s,t)\textrm{d}s, \, \int_{x}^{l}T(s,t)\textrm{d}s\big>=\frac{1}{2}\frac{\textrm{d}}{\textrm{d}t}\big\|\int_{x}^{l}T(s,t)\textrm{d}s\big\|^{2}$, we achieve
\begin{eqnarray}\label{eq.I}
\frac{\textrm{d}}{\textrm{d}t}\frac{\rho c}{2}\big\|\int_{x}^{l}T(s,t)\textrm{d}s\big\|^{2}
+\underbrace{\tau_{q}\big<q_{t}, \,\int_{x}^{l}T(s,t)\textrm{d}s\big>-\mu^{2}\big<q_{x}, \,T\big>}_{I:=}+\big(\mu^{2}q_{x}(0,t)-kT(0,t)\big)\int^{l}_{0}T(s,t)\textrm{d}s
+k\|T\|^{2}=0.
\end{eqnarray}
Exploiting the identity $\big<q_{t}, \,\int_{x}^{l}T(s,t)ds\big>=\frac{\textrm{d}}{\textrm{d}t}\big<q, \,\int_{x}^{l}T(s,t)ds\big>-\big<q, \,\int_{x}^{l}T_{t}(s,t)\textrm{d}s\big>$ together with the Eq. (\ref{Eq1}), it yields
\begin{eqnarray}
\nonumber I&:=&\tau_{q}\big<q_{t}, \,\int_{x}^{l}T(s,t)\textrm{d}s\big>-\mu^{2}\big<q_{x}, \,T\big>
\\
\nonumber&=&\frac{\textrm{d}}{\textrm{d}t}\tau_{q}\big<q, \,\int_{x}^{l}T(s,t)\textrm{d}s\big>-\tau_{q}\big<q, \,\int_{x}^{l}T_{t}(s,t)\textrm{d}s\big>+\mu^{2}\rho c\big<T_{t}, \, T\big>
\\
\nonumber&=&\frac{\textrm{d}}{\textrm{d}t}\bigg(\tau_{q}\big<q, \,\int_{x}^{l}T(s,t)\textrm{d}s\big>+\frac{\rho c}{2}\mu^{2}\|T\|^{2}\bigg)-\frac{\tau_{q}}{\rho c}\|q\|^{2}.
\end{eqnarray}	
Substituting $I$ into $(\ref{eq.I})$ we obtain
\begin{eqnarray}\label{F}
\frac{\textrm{d}}{\textrm{d}t}F(t)=-k\|T\|^{2}+\frac{\tau_{q}}{\rho c}\|q\|^{2}+C_{T}(t),
\end{eqnarray}
in which $C_{T}(t):=\big(\mu^{2}q_{x}(0,t)-kT(0,t)\big)\int_{0}^{l}T(s,t)\textrm{d}s$, and
\begin{eqnarray}
F(t):=\frac{\rho c}{2}\big\|\int_{x}^{l}T(s,t)\textrm{d}s\big\|^{2}+\frac{\rho c}{2}\mu^{2}\|T\|^{2}
+\tau_{q}\big<q, \,\int_{x}^{l}T(s,t)\textrm{d}s\big>.
\end{eqnarray}

Now, let us consider the following Lyapunov functional,
\begin{eqnarray}
\mathcal{L}(t):=\Big(2l^2+2\mu^{2}+\frac{\tau_{q}k}{\rho c}\Big)E(t)+F(t),\quad \mbox{for all} \quad t\geq 0.
\end{eqnarray}
Utilizing Cauchy–Schwarz and Young inequalities, we obtain
\begin{eqnarray}
\nonumber\Big|\mathcal{L}(t)-\Big(2l^2+2\mu^{2}+\frac{\tau_{q}k}{\rho c}\Big)E(t)\Big|&\leq&\frac{\rho c}{2}\big\|\int_{x}^{l}T(s,t)\textrm{d}s\big\|^{2}+\frac{\rho c}{2}\mu^{2}\|T\|^{2}
+\tau_{q}\Big|\big<q, \,\int_{x}^{l}T(s,t)\textrm{d}s\big>\Big|
\\
\nonumber&\leq&\frac{\rho c}{2}\big\|\int_{x}^{l}T(s,t)\textrm{d}s\big\|^{2}+\frac{\rho c}{2}\mu^{2}\|T\|^{2}
+\frac{\tau_{q}}{4\varepsilon}\|q\|^{2}+\tau_{q}\varepsilon\big\|\int_{x}^{l}T(s,t)\textrm{d}s\big\|^{2},
\end{eqnarray}
for all $\varepsilon>0$. Choosing $\varepsilon:=k/2l^{2}$ and applying Wirtinger's inequality (see Corollary 3, \cite{Swanson}), it leads to the following inequality:
\begin{eqnarray}
\nonumber\Big|\mathcal{L}(t)-\Big(2l^2+2\mu^{2}+\frac{\tau_{q}k}{\rho c}\Big)E(t)\Big|&\leq&\Big(\frac{\rho c}{2}l^2+\frac{\rho c}{2}\mu^{2}+\frac{\tau_{q}k}{2}\Big)\|T\|^{2}
+\frac{\tau_{q}}{2k}l^{2}\|q\|^{2}
\\
\nonumber&\leq&\Big(l^2+\mu^{2}+\frac{\tau_{q}k}{\rho c}\Big)\frac{\rho c}{2}\|T\|^{2}
+\frac{\tau_{q}}{2k}l^{2}\|q\|^{2}
\\
&\leq&\Big(l^2+\mu^{2}+\frac{\tau_{q}k}{\rho c}\Big)E(t).
\end{eqnarray}
Consequently, 
\begin{eqnarray}\label{auxil1}
\Big(l^2+\mu^{2}\Big)E(t)\leq \mathcal{L}(t)\leq \Big(3l^2+3\mu^{2}+2\frac{\tau_{q}k}{\rho c}\Big)E(t),\quad \mbox{for all} \quad t\geq 0.
\end{eqnarray}
Using Eqs.~\eqref{dissipation.law} and (\ref{F}), we have 
\begin{eqnarray}
\frac{\textrm{d}}{\textrm{d}t}\mathcal{L}(t)
\leq-k\|T\|^{2}-\frac{2}{k}\big(l^2+\mu^{2}\big)\|q\|^{2}+|C_{T}(t)|.
\end{eqnarray}
{\color{magenta}We choose} $\beta:=\min\big\{2k/\rho c, \, 4(l^{2}+\mu^{2})/\tau_{q}\big\}$ such that,
\begin{eqnarray}
\nonumber\frac{\textrm{d}}{\textrm{d}t}\mathcal{L}(t)
\leq-\beta E(t)+|C_{T}(t)|.
\end{eqnarray}
Utilizing the second inequality in \eqref{auxil1}, we obtain
\begin{eqnarray}\label{est.L}
\frac{\textrm{d}}{\textrm{d}t}\mathcal{L}(t)
\leq-\omega\mathcal{L}(t)+|C_{T}(t)|,\quad \mbox{for all} \quad t\geq 0,
\end{eqnarray}
where $\omega:=\frac{\beta}{3l^2+3\mu^{2}+2\tau_{q}k/\rho c}$. Multiplying $(\ref{est.L})$ by $e^{\omega t}$ yields
\begin{eqnarray}\label{L.01}
\frac{\textrm{d}}{\textrm{d}t}\Big(\mathcal{L}(t)e^{\omega t}\Big)
\leq|C_{T}(t)|e^{\omega t}.
\end{eqnarray}
Integrating (\ref{L.01}) over $[0,t]$,
\begin{eqnarray}
\mathcal{L}(t)\leq\mathcal{L}(0)e^{-\omega t}+\int_{0}^{t}|C_{T}(r)|e^{-\omega (t-r)}\textrm{d}r,\quad \mbox{for all} \quad t\geq 0,
\end{eqnarray}
where $r\in[0,t]$. Choosing  $M:=\frac{3l^2+3\mu^{2}+2\tau_{q}k/\rho c}{l^2+\mu^{2}}>1$,   $\gamma_{0}:=(l^2+\mu^{2})^{-1}$  and considering \eqref{auxil1}, we finally obtain
\begin{eqnarray}\label{est.energy}
E(t)\leq ME(0)e^{-\omega t}+\gamma_{0}\int_{0}^{t}|C_{T}(r)|e^{-\omega(t-r)}\textrm{d}r.
\end{eqnarray}
Since $C_{T}\in L^{\infty}\big((0,\infty), L^{2}(0,l)\big)$ and $T_{0}\notin L_{*}^{2}(0,l)$ we have
\begin{eqnarray}
\nonumber E(t) &\leq&  ME(0)e^{-\omega t}+(1-e^{-\omega t})\frac{\gamma_{0}}{\omega}\|C_{T}\|_{\infty}
\\
&\leq&  ME(0)e^{-\omega t}+2\frac{\gamma_{0}}{\omega}\|C_{T}\|_{\infty}.
\end{eqnarray}
This proves $(i)$. On the other hand, if $T_{0}\in L_{*}^{2}(0,l)$,  implies that $|C_{T}(t)|=0$. From (\ref{est.energy}), it immediately leads to the proof of $(ii)$.
\end{proof}

\begin{remark}
The dynamics of the system \eqref{Eq1}--\eqref{CInicial} determined by the GK equation, can be reduced to the dynamics of the system determined by the Fourier equation when
$(\tau_{q}, \mu^{2})\rightarrow (0,0)$, \textit{i.e.},
\begin{align}
\rho c T_{t}+q_{x}&=0 \quad \mbox{in} \quad (0,l) \times (0,\infty), 
\\
q+kT_{x}&=0 \quad \mbox{in} \quad (0,l) \times (0,\infty).
\end{align}
This means that the results of Theorems \ref{WP} and \ref{main.result.exponential.decay} are still valid in the limit as $(\tau_{q}, \mu^{2})\rightarrow (0,0)$.
\end{remark}
}

\section{A fully discrete finite difference model}
In this section, we consider an  implicit time integration method by finite differences applied to \eqref{Eq1}--\eqref{CInicial} system { in $ (0, l) \times (0,\mathbf{T})$}. We are mainly concerned with showing that numerical energy preserves the properties of the energy functional \eqref{energy} from being positive and non-increasing. For this reason, consider $ J,N \in \mathbb{N} $. 

We set $ \Delta x=\displaystyle \frac{l}{J+1}$  and $  \Delta t=\displaystyle \frac{\mathbf{T}}{N+1} $ and introduce the nets
\begin{eqnarray}
&&0=x_{0}<x_{1}=\Delta x < ...< x_{j}=j\Delta x<x_{J}<x_{J+1}=(J+1)\Delta x=l, 
\\
&&0=t_{0}<t_{1}=\Delta t < ...< t_{n}=n\Delta t<t_{N}<t_{N+1}=(N+1)\Delta t=\mathbf{T}.
\end{eqnarray}

We consider the following finite-diference discretization of (\ref{Eq1})--(\ref{CInicial}). For $ j=1,2,...,J $ and $ n=1,2,...,N$ we have
\begin{eqnarray}
&&\rho c \frac{T_{j}^{n}-T_{j}^{n-1}}{\Delta t}+\frac{q_{j+1}^{n}-q_{j}^{n}}{\Delta x}=0, \label{Eq.num1}
\\
&&\tau_{q}\frac{q_{j}^{n}-q_{j}^{n-1}}{\Delta t}+q_{j}^{n}-\mu^{2}\frac{q_{j+1}^{n}-2q_{j}^{n}+q_{j-1}^{n}}{\Delta x^{2}}+k\frac{T_{j}^{n}-T_{j-1}^{n}}{\Delta x}=0.  \label{Eq.num2}
\end{eqnarray}
We denote by $(T_{j}^{n}, q_{j}^{n})$ the numerical approximation for the solutions $(T, q)$ in points $(x_{j},t_{n})$ of the mesh. For boundary conditions, we use
\begin{eqnarray}\label{CCDiscretEqDiscret.Total}
q_{0}^{n}=q_{J+1}^{n}=0, \quad \mbox{for all} \quad n=0,1,...,N+1 
\end{eqnarray}
and for the for the initial conditions, we adopted
\begin{eqnarray}
T_{j}^{0}=T_{0j}, \quad q_{j}^{0}=q_{0j}, \quad \mbox{for all} \quad j=0,1,...,J+1. \label{CIDiscretEqDiscret.Total.}
\end{eqnarray}

\subsection{Discrete functional energy}
In this subsection we analyze in detail some properties of the energy of the discrete system   (\ref{Eq.num1})--(\ref{CIDiscretEqDiscret.Total.}). We  define the functional energy of (\ref{Eq.num1})--(\ref{CIDiscretEqDiscret.Total.}) by
\begin{eqnarray}\label{numerical.energy}
\displaystyle\mathbf{E^{n}}:=\rho c\frac{\Delta x}{2}\sum_{j=0}^{J}|T_{j}^{n}|^{2}+\frac{\tau_{q}}{k}\frac{\Delta x}{2}\sum_{j=0}^{J}|q_{j}^{n}|^{2}, \quad \mbox{for all} \quad n=1,...,N+1.
\end{eqnarray}

\begin{theorem}\label{teo.energy.law}
Let $(T_{j}^{n}, q_{j}^{n})$ the numerical solution of the problem (\ref{Eq.num1})--(\ref{CIDiscretEqDiscret.Total.}). Thus, for all $ \Delta x,\Delta t \in (0,1) $ the rate of change of the numerical energy $ \mathbf{E^{n}} $ in $ (\ref{numerical.energy}) $ at the instant $ t_{n} $ is given by
\begin{eqnarray}
\frac{ \mathbf{E^{n}}- \mathbf{E^{n-1}}}{\Delta t}\leq-\frac{1}{k}\Delta x\sum_{j=0}^{J}|q_{j}^{n}|^{2}-\frac{\mu^{2}}{k}\Delta x\sum_{j=0}^{J}\bigg|\frac{q_{j+1}^{n}-q_{j}^{n}}{\Delta x}\bigg|^{2}< 0,
\end{eqnarray}
for all $n=1,2,...,N+1$.
\end{theorem}
\begin{proof}
Multiplying the Eq. $(\ref{Eq.num1})$ by $\Delta xT_{j}^{n}$ and adding to $ j=0,1,...,J $ we have
\begin{eqnarray}
\rho c\Delta x \sum_{j=0}^{J}\bigg(\frac{T_{j}^{n}-T_{j}^{n-1}}{\Delta t}T_{j}^{n}\bigg)+\Delta x \sum_{j=0}^{J}\bigg(\frac{q_{j+1}^{n}-q_{j}^{n}}{\Delta x}T_{j}^{n}\bigg)=0.
\end{eqnarray}
Using Young's inequality we have
\begin{eqnarray}\label{desig.01}
\rho c\frac{\Delta x}{2\Delta t} \sum_{j=0}^{J}|T_{j}^{n}|^{2}-\rho c\frac{\Delta x}{2\Delta t} \sum_{j=0}^{J}|T_{j}^{n-1}|^{2}+\Delta x \sum_{j=0}^{J}\bigg(\frac{q_{j+1}^{n}-q_{j}^{n}}{\Delta x}T_{j}^{n}\bigg)\leq0.
\end{eqnarray}
Similarly, multiplying the Eq. $(\ref{Eq.num2})$ by $k^{-1}\Delta xq_{j}^{n}$, adding to $ j=1,...,J $ and using the boundary conditions $q_{0}^{n}=0$ we have
\begin{eqnarray}\label{desig.02}
&&\nonumber\frac{\tau_{q}}{k}\frac{\Delta x}{2\Delta t} \sum_{j=0}^{J}|q_{j}^{n}|^{2}-\frac{\tau_{q}}{k}\frac{\Delta x}{2\Delta t} \sum_{j=0}^{J}|q_{j}^{n-1}|^{2}+\frac{1}{k}\Delta x\sum_{j=0}^{J}|q_{j}^{n}|^{2}
\\
&&\underbrace{-\frac{\mu^{2}}{k}\Delta x\sum_{j=1}^{J}\bigg(\frac{q_{j+1}^{n}-2q_{j}^{n}+q_{j-1}^{n}}{\Delta x^{2}}q_{j}^{n}\bigg)+\Delta x \sum_{j=1}^{J}\bigg(\frac{T_{j}^{n}-T_{j-1}^{n}}{\Delta x}q_{j}^{n}\bigg)}_{J:=}\leq0.
\end{eqnarray}

Using the boundary conditions (\ref{CCDiscretEqDiscret.Total}) we have
\begin{eqnarray}\label{J}
\nonumber J&:=&-\frac{\mu^{2}}{k}\Delta x\sum_{j=1}^{J}\bigg(\frac{q_{j+1}^{n}-2q_{j}^{n}+q_{j-1}^{n}}{\Delta x^{2}}q_{j}^{n}\bigg)+\Delta x \sum_{j=1}^{J}\bigg(\frac{T_{j}^{n}-T_{j-1}^{n}}{\Delta x}q_{j}^{n}\bigg)
\\
&=&\frac{\mu^{2}}{k}\Delta x\sum_{j=0}^{J}\bigg|\frac{q_{j+1}^{n}-q_{j}^{n}}{\Delta x}\bigg|^{2}-\Delta x \sum_{j=0}^{J}\bigg(T_{j}^{n}\frac{q_{j+1}^{n}-q_{j}^{n}}{\Delta x}\bigg).
\end{eqnarray}
From (\ref{desig.02}) and (\ref{J}) we obtain
\begin{eqnarray}\label{desig.03}
&&\nonumber\frac{\tau_{q}}{k}\frac{\Delta x}{2\Delta t} \sum_{j=0}^{J}|q_{j}^{n}|^{2}-\frac{\tau_{q}}{k}\frac{\Delta x}{2\Delta t} \sum_{j=0}^{J}|q_{j}^{n-1}|^{2}+\frac{1}{k}\Delta x\sum_{j=0}^{J}|q_{j}^{n}|^{2}+\frac{\mu^{2}}{k}\Delta x\sum_{j=0}^{J}\bigg|\frac{q_{j}^{n+1}-q_{j}^{n}}{\Delta x}\bigg|^{2}
\\
&&-\Delta x \sum_{j=0}^{J}\bigg(T_{j}^{n}\frac{q_{j+1}^{n}-q_{j}^{n}}{\Delta x}\bigg)\leq0.
\end{eqnarray}
Adding the inequalities (\ref{desig.01}) and (\ref{desig.03})  and using the energy $ \mathbf{E^{n}} $ in $ (\ref{numerical.energy}) $ we obtain
\begin{eqnarray}
\frac{\mathbf{E^{n}}-\mathbf{E^{n-1}}}{\Delta t}\leq-\frac{1}{k}\Delta x\sum_{j=0}^{J}|q_{j}^{n}|^{2}-\frac{\mu^{2}}{k}\Delta x\sum_{j=0}^{J}\bigg|\frac{q_{j+1}^{n}-q_{j}^{n}}{\Delta x}\bigg|^{2},
\end{eqnarray}
and therefore $\mathbf{E^{n}}\leq \mathbf{E^{0}}$ for all $n=1,2,...,N+1$. 
\end{proof}

\section{Numerical simulation}
Let us rewrite the scheme (\ref{Eq.num1})--(\ref{CIDiscretEqDiscret.Total.}) in an equivalent vectorial form. We denote by $\textbf{L}$ and $\textbf{I}_{m}$ the tridiagonal
matrix associated to the approximation of the Laplacian  and $m\times m$ identity matrix respectively, \textit{i.e.}, 
$$
\textbf{L}:=\left(\begin{array}{cccccc}
-2&  1&  0&  0&  \cdots& 0  \\
1&  -2&  1&  \ddots&  \ddots& \colon  \\
0&  1&  \ddots&  \ddots&  \ddots& 0  \\
0&  \ddots&  \ddots&  \ddots&  1& 0  \\
\colon&  \ddots&  \ddots&  1&  -2& 1 \\
0&  \cdots&  0&  0&  1& -2  
\end{array}\right)_{J\times J}, 
\qquad
\textbf{I}_{m}:=\left(\begin{array}{cccccc}
1&  0&  0&  0&  \cdots& 0  \\
0&  1&  0&  \ddots&  \ddots& \colon  \\
0&  0&  \ddots&  \ddots&  \ddots& 0  \\
0&  \ddots&  \ddots&  \ddots&  0& 0  \\
\colon&  \ddots&  \ddots&  0&  1& 0 \\
0&  \cdots&  0&  0&  0& 1  
\end{array}\right)_{m\times m}. 
$$
On the other hand, $\textbf{A}_{q}$ and $\textbf{A}_{T}$ are  trapezoidal matrices given by
$$
\textbf{A}_{q}:=\left(\begin{array}{cccccc}
1&  0&  0&  0&  \cdots& 0  \\
-1&  1&  0&  \ddots&  \ddots& \colon  \\
0&  -1&  \ddots&  \ddots&  \ddots& 0  \\
0&  \ddots&  \ddots&  \ddots&  0& 0  \\
\colon&  \ddots&  \ddots&  -1&  1& 0 \\
0&  \cdots&  0&  0&  -1& 1 \\
0&  \cdots&  0&  0&  0& -1   
\end{array}\right)_{J+1\times J}, \qquad \textbf{A}_{T}:=\left(\begin{array}{ccccccc}
-1&  1&  0&  0&  \cdots& 0  & 0 \\
0&  -1&  1&  \ddots&  \ddots& \colon  & 0   \\
0&  0&  \ddots&  \ddots&  \ddots& 0  & 0   \\
0&  \ddots&  \ddots&  \ddots&  1& 0  & 0   \\
\colon&  \ddots&  \ddots&  0&  -1& 1  & 0  \\
0&  \cdots&  0&  0&  0& -1  & 1    
\end{array}\right)_{J\times J+1}.
$$
Hence we define the mass matrices
\begin{eqnarray}
\textbf{B}:=\textbf{I}_{J}-\frac{\mu^{2}}{\tau_{q}+\Delta t}\frac{\Delta t}{\Delta x^{2}}\textbf{L}, \qquad \textbf{C}:=\textbf{I}_{J+1}+\frac{k}{\rho c(\tau_{q}+\Delta t)}\frac{\Delta t}{\Delta x^{2}}\textbf{A}_{q}\textbf{B}^{-1}\textbf{A}_{T}, \qquad \textbf{D}:=\textbf{A}_{q}\textbf{B}^{-1}.
\end{eqnarray}

If we denote $\mathbb{T}^{n}=(T_{0}^{n}, T_{1}^{n},...,T_{J}^{n})^{\top}$ and $\mathbb{Q}^{n}=(q_{1}^{n}, q_{2}^{n},...,q_{J}^{n})^{\top}$, then  scheme (\ref{Eq.num1})--(\ref{CIDiscretEqDiscret.Total.}) takes the following vectorial form:
\begin{eqnarray}\label{numerical.problem}
\begin{cases}
\displaystyle \mathbb{T}^{n}=\textbf{C}\mathbb{T}^{n-1}-\frac{\tau_{q}}{\rho c(\tau_{q}+\Delta t)}\frac{\Delta t}{\Delta x}\textbf{D}\mathbb{Q}^{n-1}, \quad n=1, 2, ...,N+1, 
\\
\displaystyle \mathbb{Q}^{n}=\frac{\tau_{q}}{\tau_{q}+\Delta t}\textbf{B}^{-1}\mathbb{Q}^{n-1}-\frac{k}{\tau_{q}+\Delta t}\frac{\Delta t}{\Delta x}\textbf{B}^{-1}\textbf{A}_{T}\mathbb{T}^{n-1}, \quad n=1, 2,...,N+1,
\\
\displaystyle \mathbb{T}^{0}=(T_{0}^{0}, T_{1}^{0},...,T_{J}^{0})^{\top},  \quad \mathbb{Q}^{0}=(q_{1}^{0}, q_{2}^{0},...,q_{J}^{0})^{\top}.
\end{cases}
\end{eqnarray}

Let us present a numerical example, especially with regard to the behavior of discrete energy $\textbf{E}^n$ given by \eqref{numerical.energy}, with aiming to confirm the analytical results established in previous sections. For our purposes, let us consider $l=1\cdot 10^{-1}$ m and for the computational mesh we take $\Delta x=2\cdot10^{-4}$ m, $\Delta t=1.2\cdot10^{-2}$ s. We use $\rho = 2\cdot 10^{3}$ kg/m$^3$, $c = 5\cdot 10^{2}$ J/(kg K), $\tau_{q} = 8\cdot 10^{-3}$ s, $\mu^{2} = 2.8\cdot 10^{-3}$ m$^2$, $k = 2 \cdot 10^{3}$ W/(m K) and we take the following initial conditions associated to the discrete system \eqref{numerical.problem}, following \cite{Kov22}:
\begin{eqnarray}\label{initial.conditions}
T_{j}^{0}:=T_{b}+\frac{T_{f}}{2}\cos\Big(\frac{\pi x_{j}}{l}\Big),  \quad {j=0, 1,...,J}.
\end{eqnarray}

{
Below we provide the simulation of the solution and the energy functional obtained from the numerical scheme (\ref{numerical.problem}), considering two cases in which we solve the Fourier and Guyer-Krumhansl heat equations in regard to the Theorem \ref{main.result.exponential.decay} (item (i)). Since Theorem \ref{main.result.exponential.decay} (item (ii)) fals beyond the physical interest, and also appears as a special case of item (i), we do not deal with this situation. Although from the thermodynamic point of view the following simulations are straightforward, we present them in order to support the mathematical proofs, which properties are also preserved in the discrete case. Naturally, the solutions of the Guyer-Krumhansl and Fourier heat equations converges to the same equilibrium state with the same energy content, that stands as a thermodynamic requirement.

{
\textbf{Case I:} We consider $T_{j}^{0}:=T_{b}+\frac{T_{f}}{2}\cos\big(\frac{\pi x_{j}}{l}\big)$ with $T_{b} = 1.5\cdot 10^{1}$ $^\circ$C and $T_{f} = 3\cdot 10^{1}$ $^\circ$C. This implies that $T_{j}^{0} \notin L_{*}^{2}(0,l)$. According to the Theorem \ref{main.result.exponential.decay} (item (i)) the energy functional decays exponentially to a positive constant. This is present in the simulation below.
}
\begin{figure}[htbp]
\centering
\begin{subfigure}{0.49\textwidth}
\includegraphics[width=\linewidth]{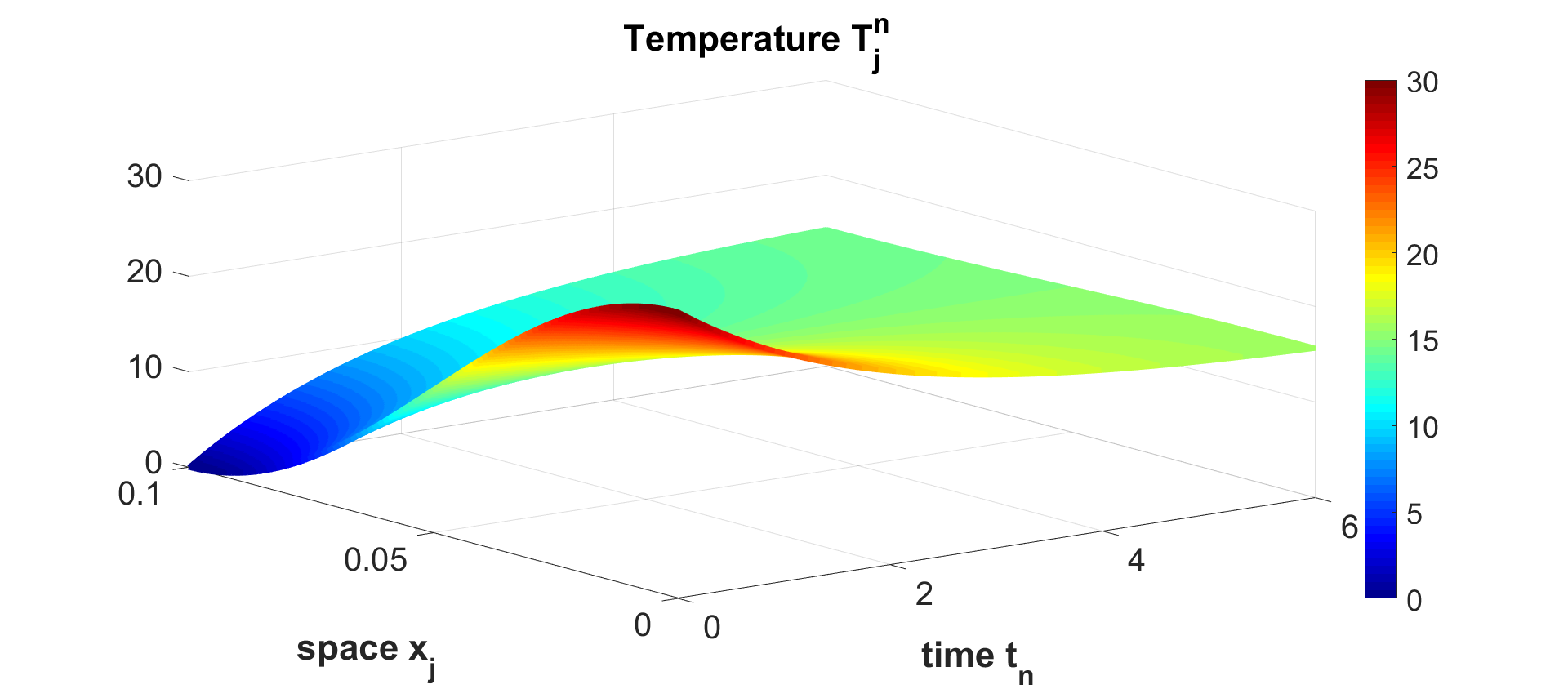}
\caption{} \label{fig:1a_GK}
\end{subfigure}
\begin{subfigure}{0.49\textwidth}
\includegraphics[width=\linewidth]{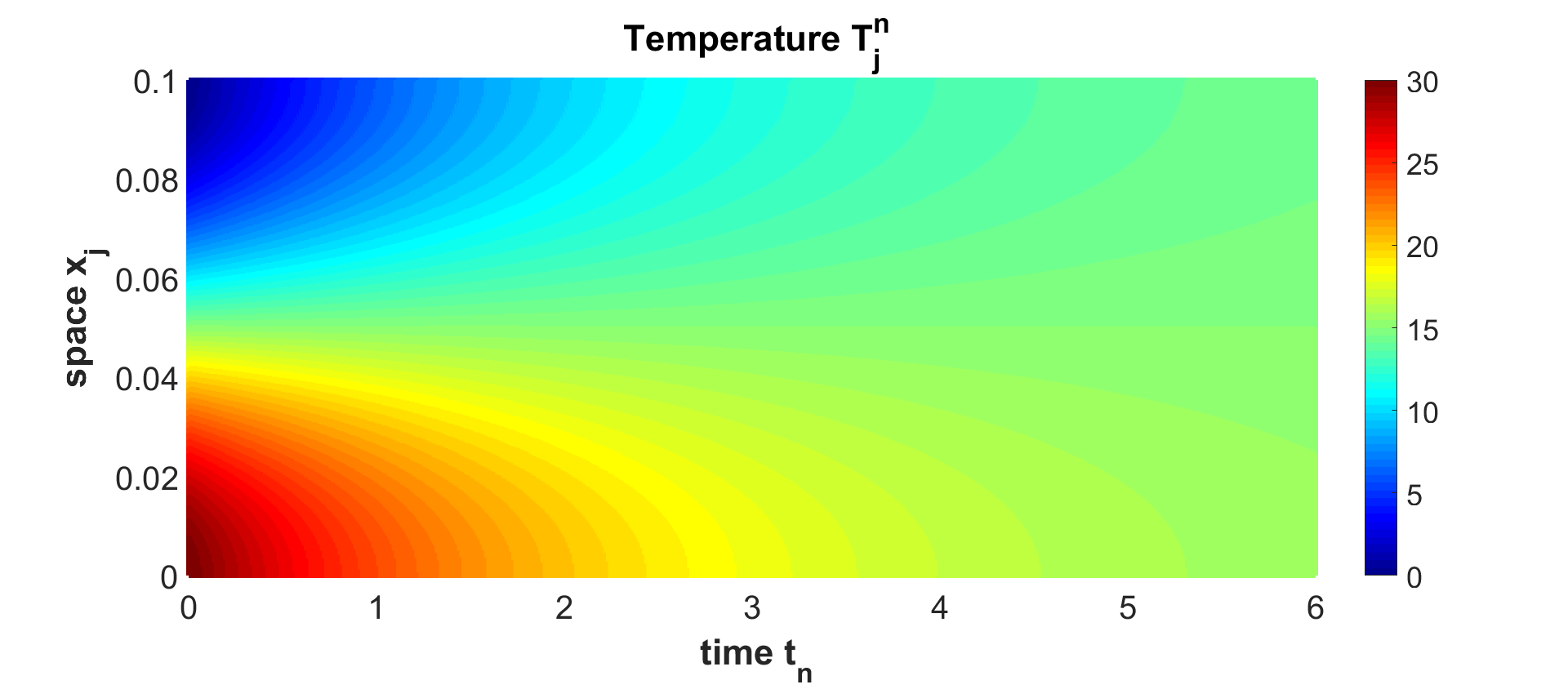}
\caption{} \label{fig:1b_GK}
\end{subfigure}
\begin{subfigure}{0.49\textwidth}
\includegraphics[width=\linewidth]{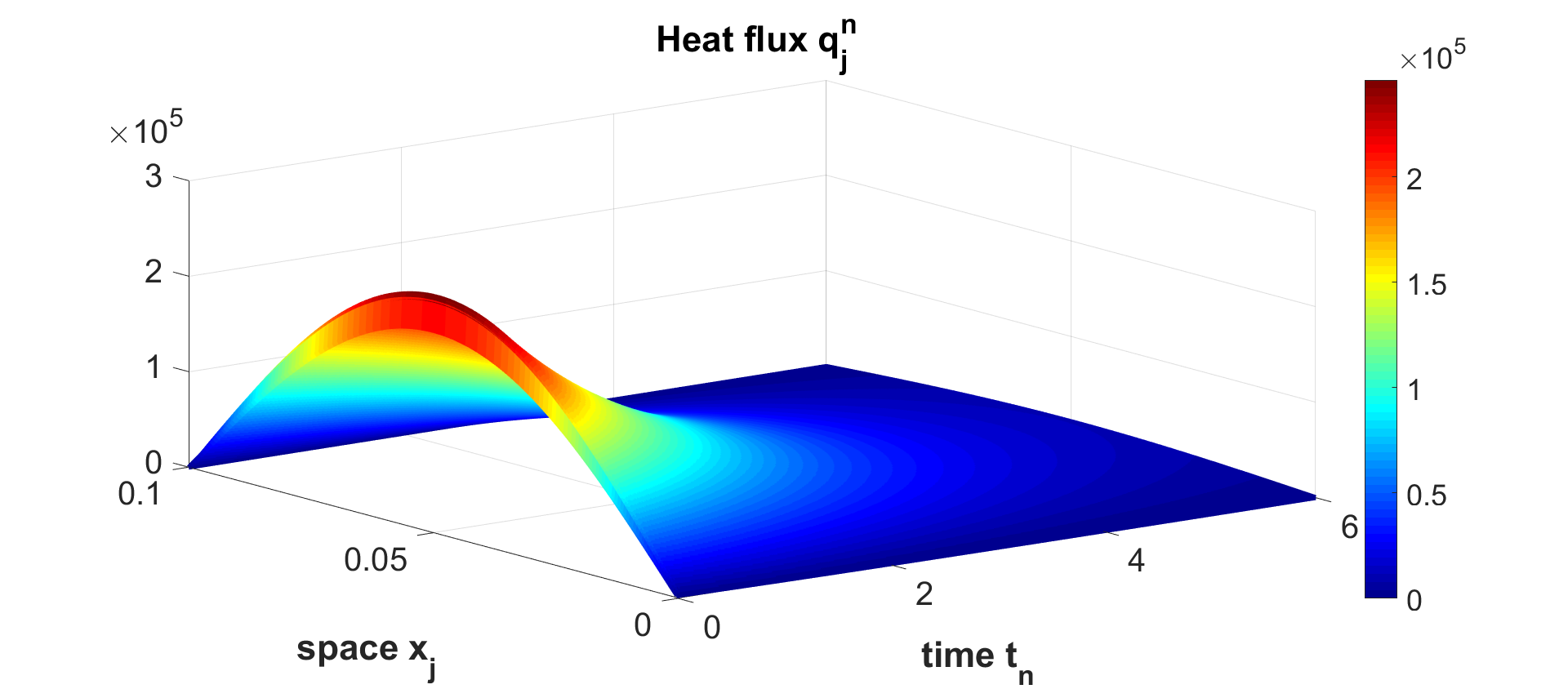}
\caption{} \label{fig:2a_GK}
\end{subfigure}
\begin{subfigure}{0.49\textwidth}
\includegraphics[width=\linewidth]{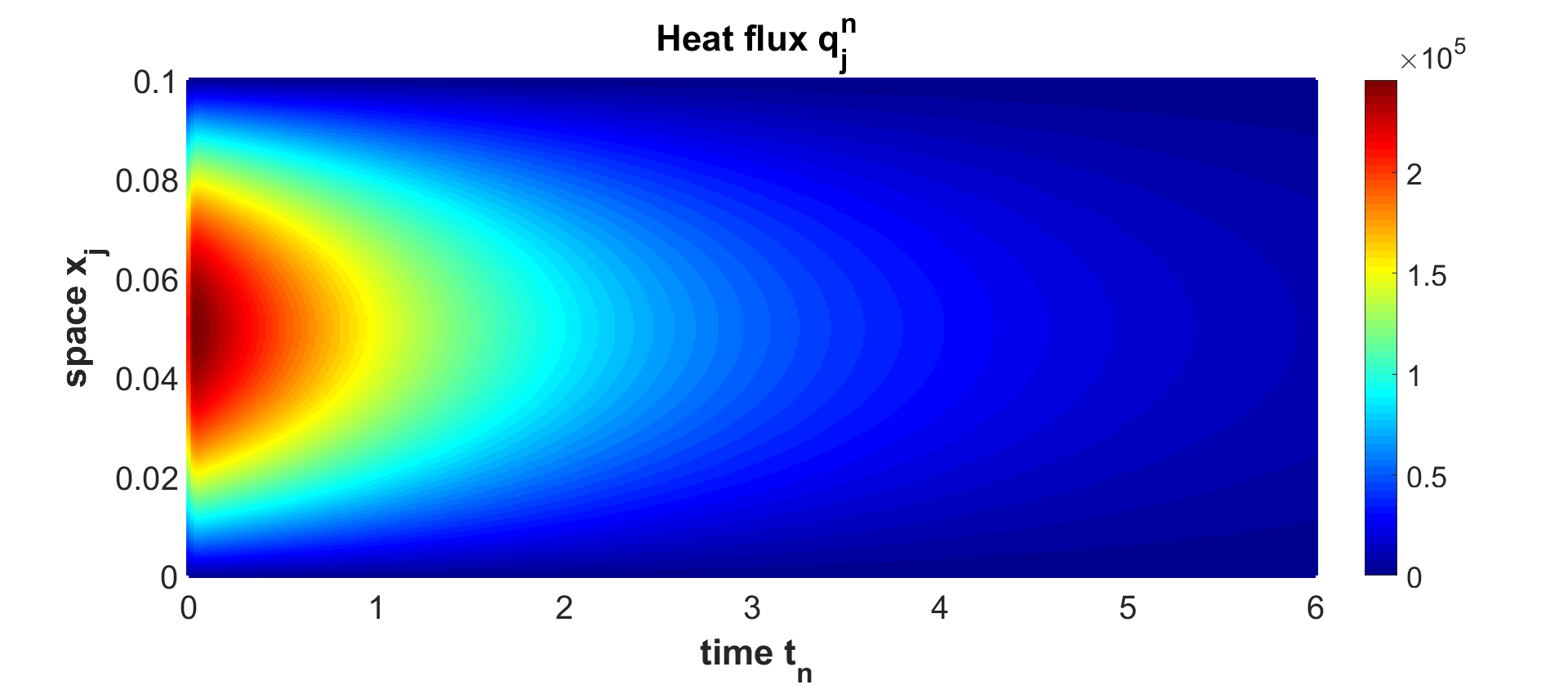}
\caption{} \label{fig:2b_GK}
\end{subfigure}
\begin{subfigure}{0.49\textwidth}
\includegraphics[width=\linewidth]{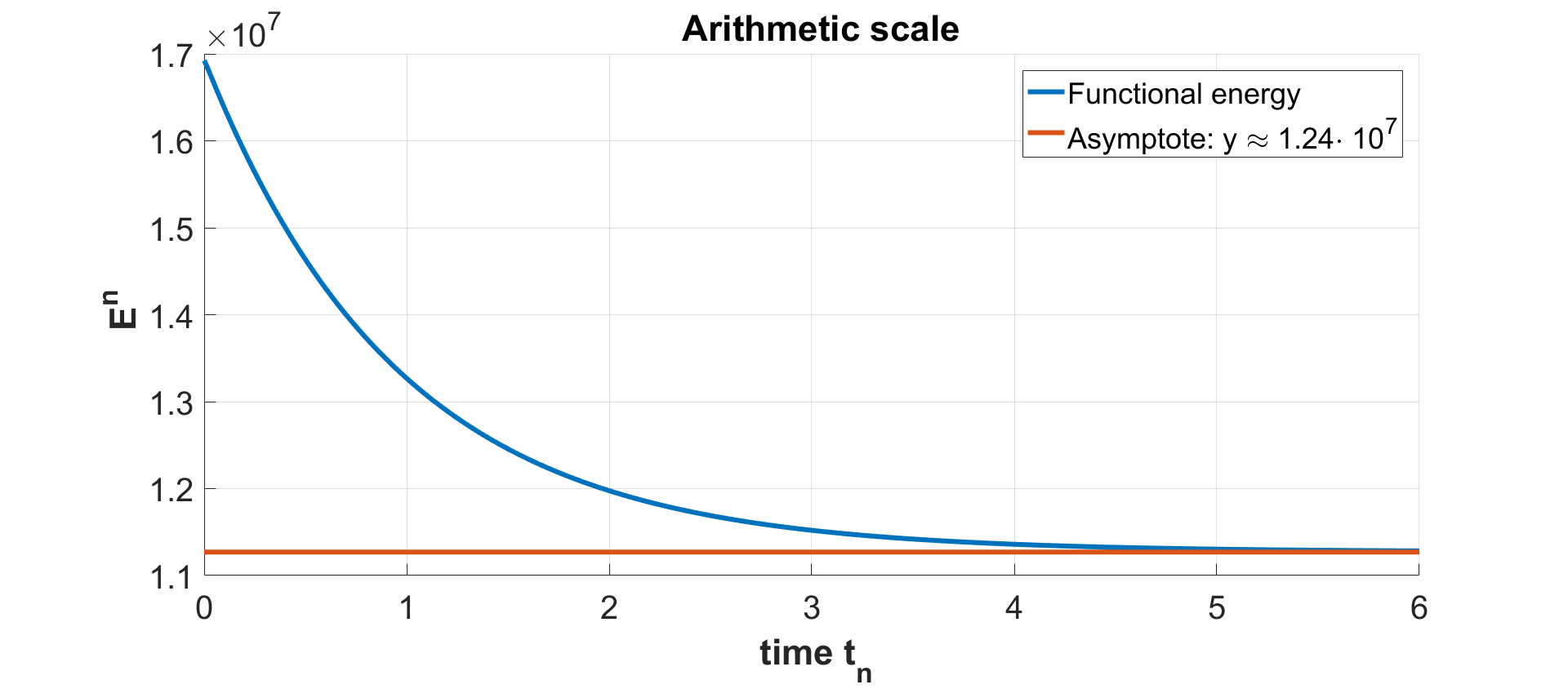}
\caption{} \label{fig:3a_GK}
\end{subfigure}
\begin{subfigure}{0.49\textwidth}
\includegraphics[width=\linewidth]{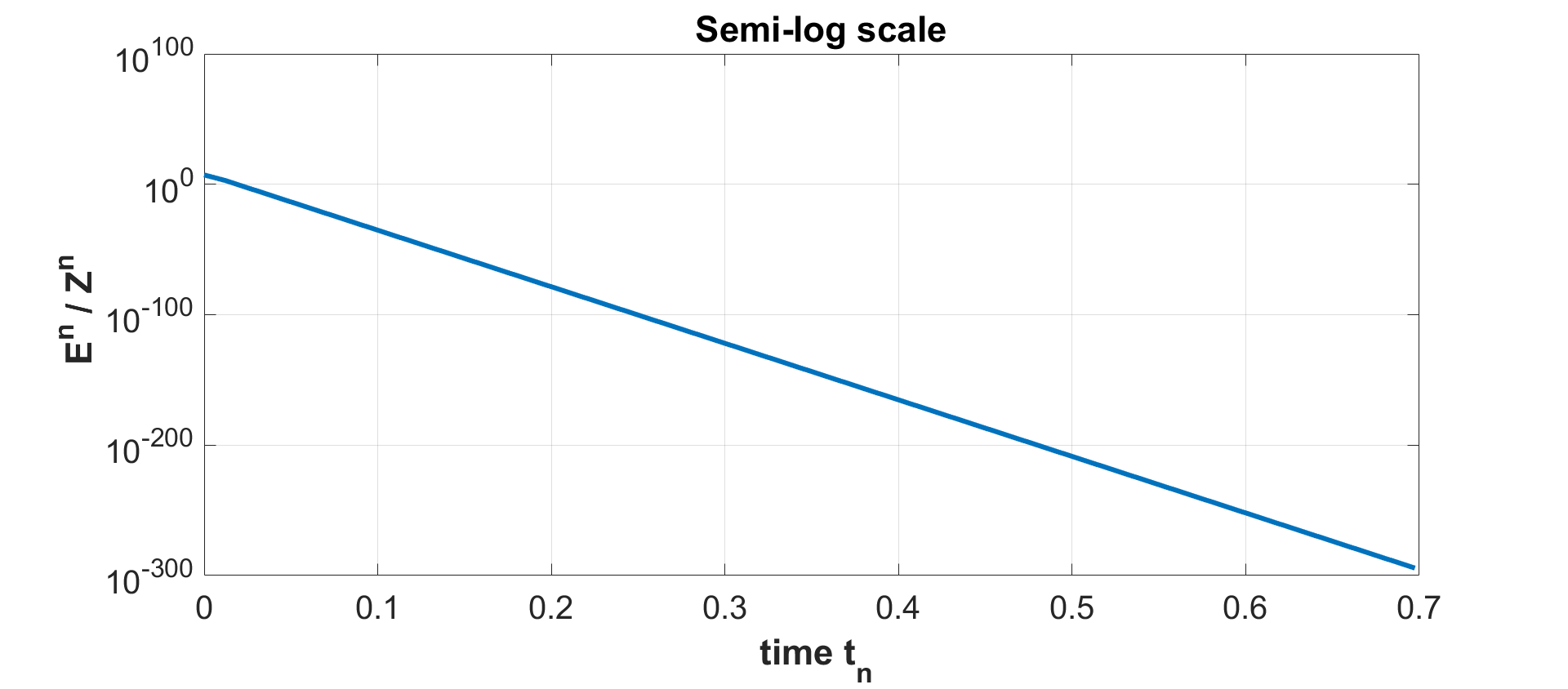}
\caption{} \label{fig:3b_GK}
\end{subfigure}
\caption{{
Guyer–Krumhansl heat equation. We use $\tau_{q} = 8\cdot 10^{-3}$s, $\mu^{2} = 2.8\cdot 10^{-3}$ m$^2$ and $Z^{n}:=1+(M_{1}\|C_{T}\|_{\infty}\big/M_{0}E^{0})e^{\omega t_{n}}$. As $T_{0}\notin L_{*}^{2}(0,l)$ the temperature history reaches an equilibrium point (see (a) and (b)), exponential decay tends to a positive constant {
$ y \approx 1.24\cdot10^{7}$}.} }
\label{fig:2}
\end{figure}

\newpage

{
}


{
The simulations for the Fourier heat equation are also obtained from the numerical scheme (\ref{numerical.problem}) by choosing $(\tau_{q}, \mu^{2})=(0,0)$. 

\textbf{Case II:} Here, we consider the data from Case I with $\tau_{q} = \mu^{2} = 0$. According to the Theorem \ref{main.result.exponential.decay} (item (i)) the energy functional decays exponentially to a positive constant. 
}

\begin{figure}[htbp]
\centering
\begin{subfigure}{0.49\textwidth}
\includegraphics[width=\linewidth]{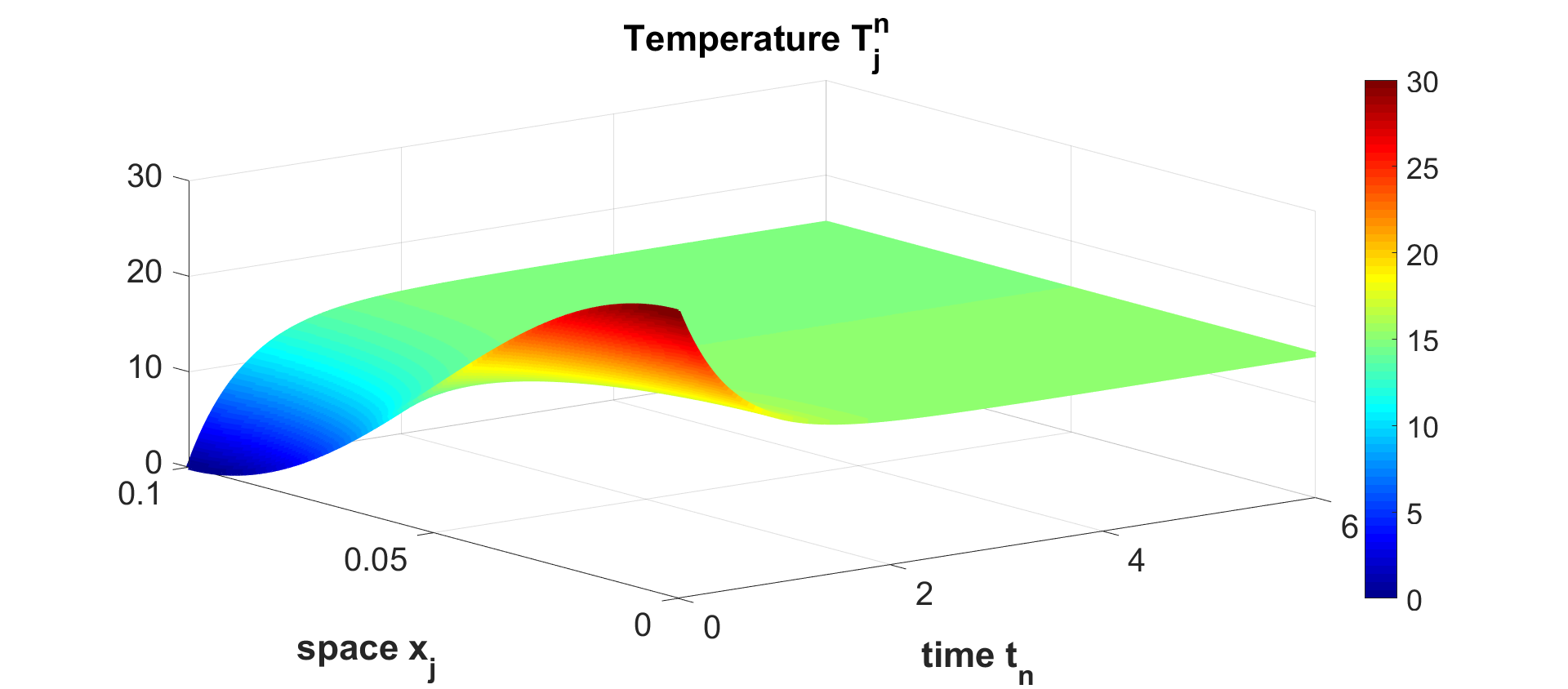}
\caption{} \label{fig:1a_FR}
\end{subfigure}
\begin{subfigure}{0.49\textwidth}
\includegraphics[width=\linewidth]{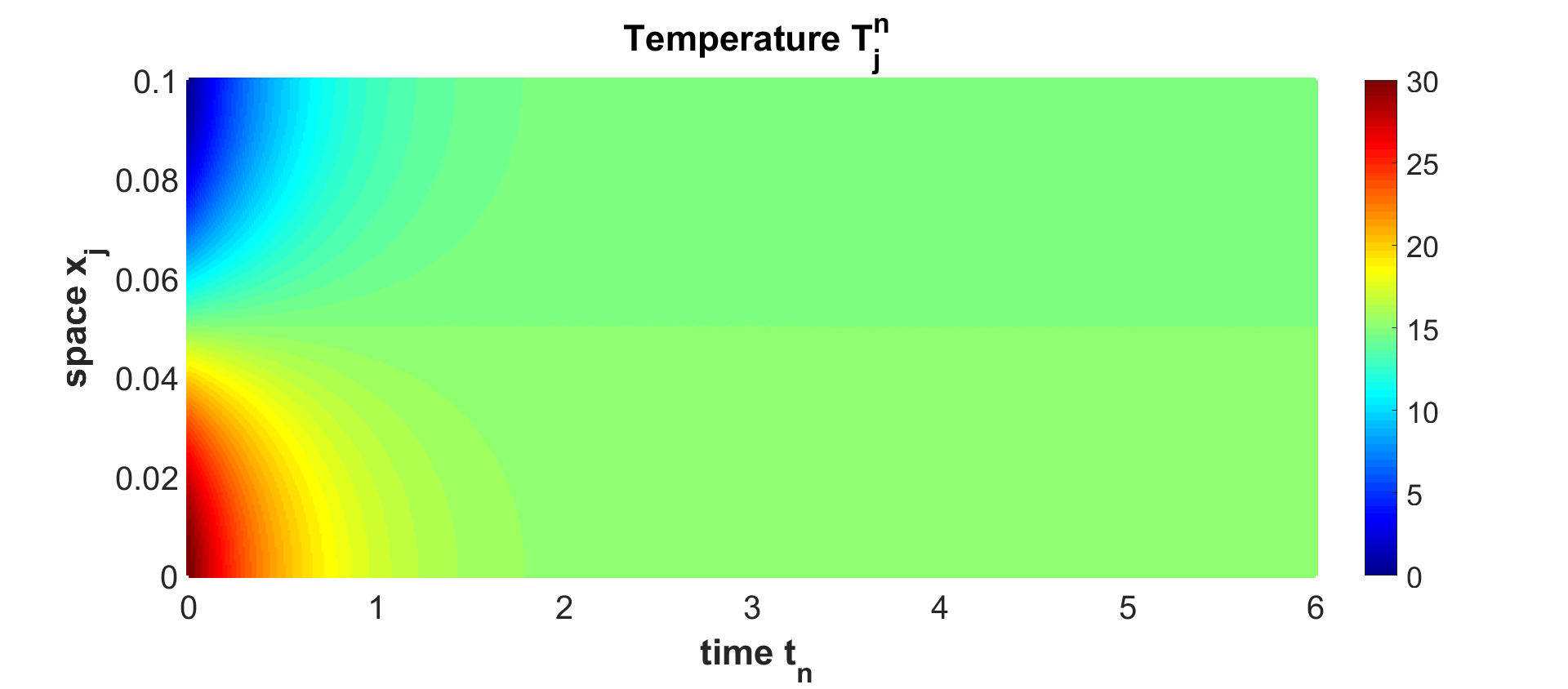}
\caption{} \label{fig:1b_FR}
\end{subfigure}
\begin{subfigure}{0.49\textwidth}
\includegraphics[width=\linewidth]{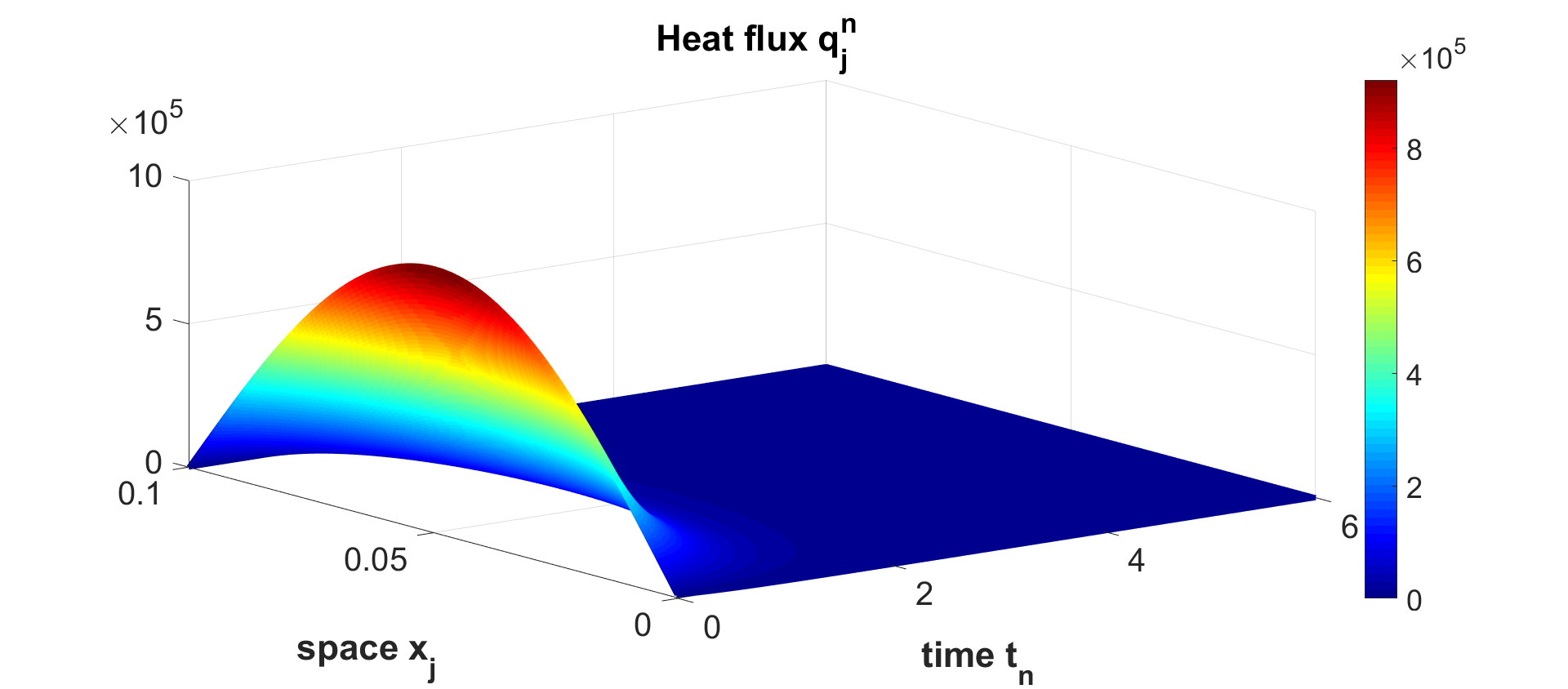}
\caption{} \label{fig:2a_FR}
\end{subfigure}
\begin{subfigure}{0.49\textwidth}
\includegraphics[width=\linewidth]{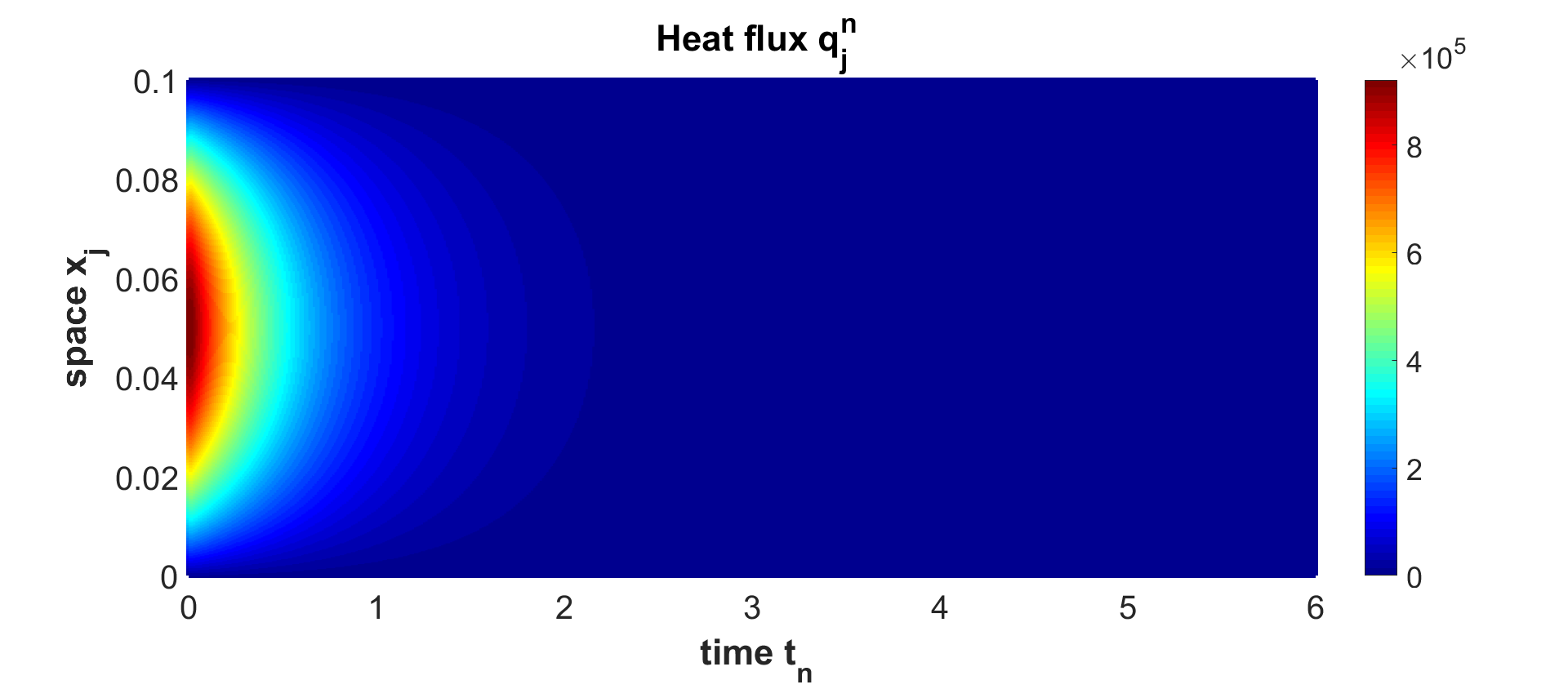}
\caption{} \label{fig:2b_FR}
\end{subfigure}
\begin{subfigure}{0.49\textwidth}
\includegraphics[width=\linewidth]{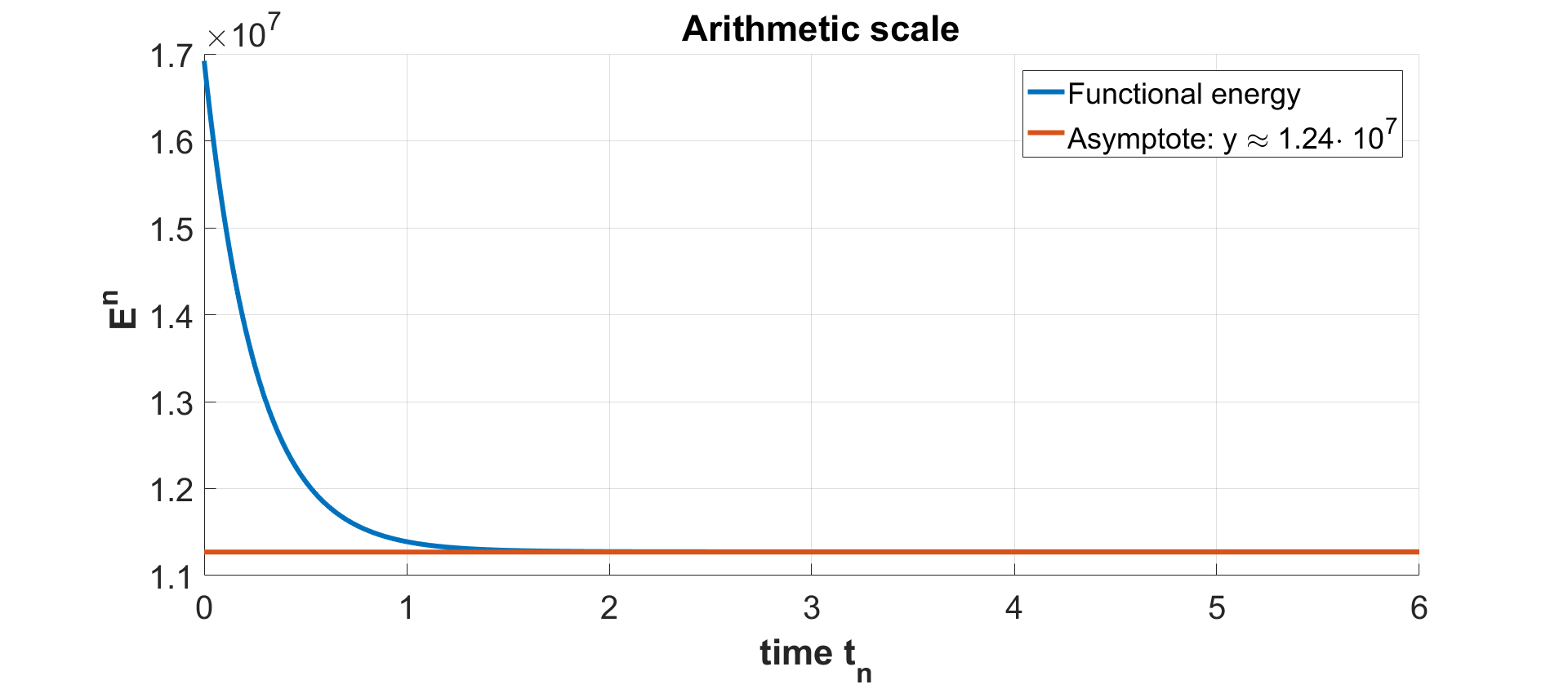}
\caption{} \label{fig:3a_FR}
\end{subfigure}
\begin{subfigure}{0.49\textwidth}
\includegraphics[width=\linewidth]{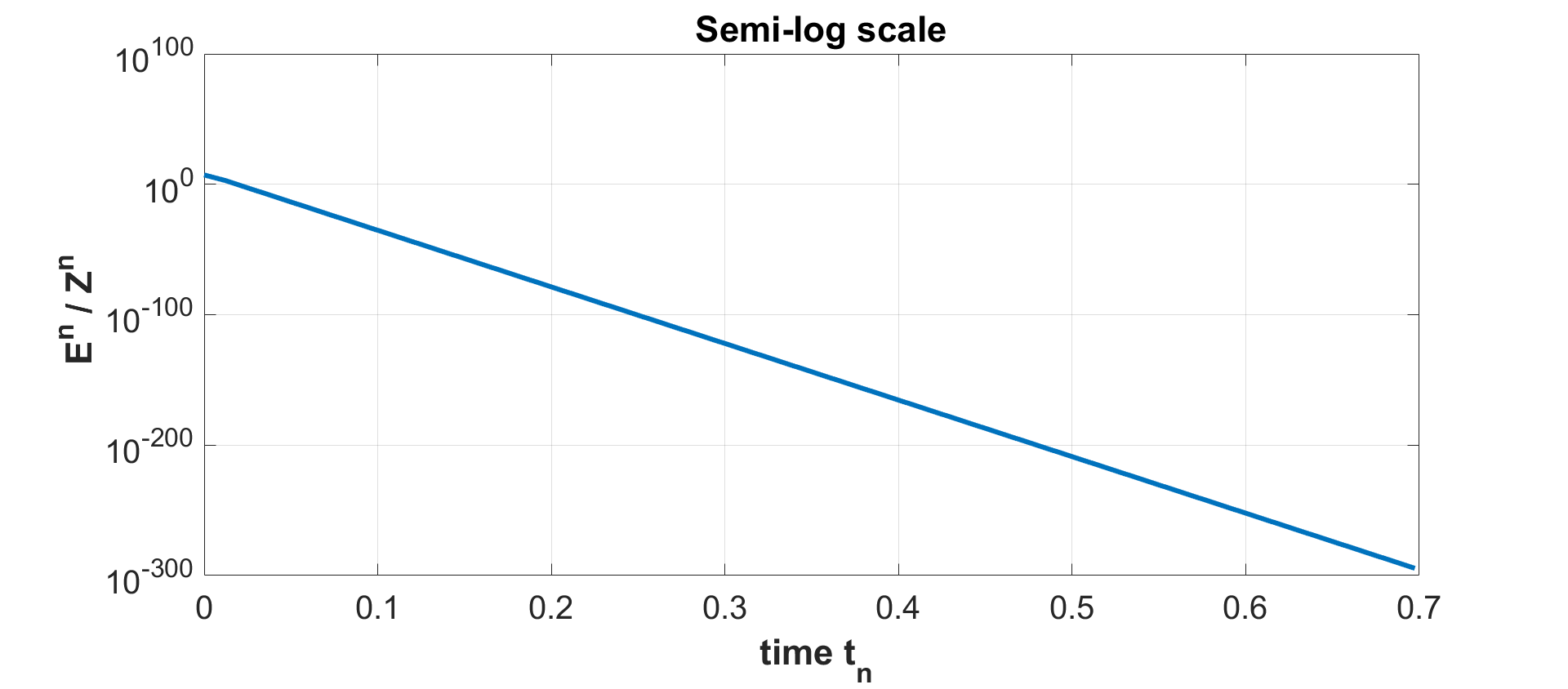}
\caption{} \label{fig:3b_FR}
\end{subfigure}
\caption{Solutions for the Fourier heat equation, presenting the agreement between the physical and mathematical requirements in regard to the Theorem \ref{main.result.exponential.decay} (item (i)).}  
\label{fig:2}
\end{figure}
}

{

}


\section{Discussion}

It is known for decades that Fourier's law has limitations, thus it is inevitable to find its proper extension and discover the possibilities for practical implementations. In this respect, the Guyer-Krumhansl equation is a promising candidate.
While the GK equation has a more complex structure than Fourier's law, its modeling capabilities can cover wide range of phenomena from low-temperature to room temperature problems. Numerous experiments show the validity of the GK equation in these cases. In order to take a step forward to have the GK equation to be the next standard model, we presented three crucial properties of that model.

First, we have seen that for temperature-dependent situations, the thermodynamic background must be exploited as it connects the coefficients and they are not independent of each other. It influences the entire model, and introduces new terms into the consitutive equation. The discovery and understanding of these new terms will help the experimental studies in the future how and in what aspect to focus on the measurements and how to evaluate them. 
{
Second, we have proved the well-posedness of the GK equation, that is, its solution exist, unique and continously depends on the initial data. This is crucial to have a strong mathematical and physical basis, and avoid the issues which are characteristic for the DPL concept or any other models lacking the proper background. Moreover, we also proved the uniform stabilization for a practically relevant case in which the initial temperature distribution $T_{0}\notin L_{*}^{2}(0,l)$, i.e., the initial energy content of the system is not zero. Our statements are demonstrated through numerical solutions, the proved properties are preserved in the discrete space as well.
These results establish a strong basis for the Guyer-Krumhansl heat equation and for its later applications. 
}
\vspace{0.5cm}

\section*{Funding}
A. J. A. Ramos was partially supported by CNPq Grant 310729/2019-0.

M. M. Freitas was partially supported by CNPq Grant 313081/2021-2.

D. S. Almeida Júnior was partially supported by CNPq Grant 314273/2020-4.

R. Kovács: The research reported in this paper is part of project no. BME-NVA-02, implemented with the support provided by the Ministry of Innovation and Technology of Hungary from the National Re-search, Development and Innovation Fund, financed under the TKP2021 funding scheme.
This paper was supported by the János Bolyai Research Scholarship of the Hungarian Academy of Sciences.
The research reported in this paper and carried out at BME has been supported by the grants National Research, Development and Innovation Office-NKFIH FK 134277.

\section*{Declarations}
\textbf{Conflict of interest} The author declares no competing interests.

\end{document}